\documentclass[11pt]{amsart}

\usepackage[english]{babel}
\usepackage[T1]{fontenc} 
\usepackage{lmodern}
\usepackage[utf8]{inputenc} 
\usepackage[final]{microtype} 
\usepackage{graphicx}
\usepackage{tikz}
\usetikzlibrary{calc}
\usepackage{tumcolors2}
\usepackage{pgfplots}

\usepackage{mathtools}

\usepackage[%
  allcolors=black,%
  breaklinks=true, 
  colorlinks=false, 
  plainpages=false,%
  hypertexnames=false,%
  bookmarksopen=false,%
  bookmarksnumbered=false%
]%
{hyperref}%
\usepackage[autostyle=true]{csquotes}
\usepackage[backend=biber ,citestyle=alphabetic, bibstyle=alphabetic,
hyperref=true, maxbibnames=20, maxcitenames=3, url=false, doi=false,
giveninits=true]{biblatex}
\usepackage{amsmath, amssymb}
\usepackage{tabularx}
\usepackage{booktabs}
\usepackage{caption}
\usepackage{pifont}
\newcommand{\cmark}{\ding{51}}
\newcommand{\xmark}{\ding{55}}

\usetikzlibrary{calc, shapes, decorations, automata, positioning, external, arrows, spy, intersections, matrix}

\usepackage{enumitem}
\setlist[enumerate,1]{label=\alph*)}
\setlist[enumerate,2]{label=(\roman*),ref=\theenumi(\roman*)}
\setlist[enumerate,3]{label=\Roman*.}

\usepackage{nicefrac}
\usepackage{bbm}
\usepackage{xspace}

\definecolor{dgreen}{HTML}{41A317}

\DeclareMathOperator{\epi}{epi}
\DeclareMathOperator{\pos}{pos}
\DeclareMathOperator{\cl}{cl}
\DeclareMathOperator{\aff}{aff}
\DeclareMathOperator{\lin}{lin}
\DeclareMathOperator{\relint}{relint}
\DeclareMathOperator{\relbd}{relbd}
\DeclareMathOperator{\conv}{conv}

\newcommand{\eg}{e.\,g.}
\newcommand{\ie}{i.\,e.}
\newcommand{\posreals}{\ensuremath{\reals_{\geq 0}}}

\newcommand{\reals}{\ensuremath{\mathbb{R}}}

\newcommand{\halfspace}[3][\le]{\ensuremath{H^{#1}_{(#2,#3)}}}
\newcommand{\hplane}[2]{\ensuremath{H_{(#1,#2)}}}

\newcommand{\CH}{\ensuremath{\mathcal{H}}}

\newcommand{\subjectto}{\text{s.t.}\ }

\newcommand*\colvec[1]{\begin{pmatrix}#1\end{pmatrix}}

\renewcommand{\phi}{\varphi}
\renewcommand{\epsilon}{\varepsilon}

   \providecommand\given{}
   \newcommand\SetSymbol[1][]{%
      \nonscript\:#1\vert
      \allowbreak
      \nonscript\:
      \mathopen{}}
   \DeclarePairedDelimiterX\Set[1]\{\}{%
      \renewcommand\given{\SetSymbol[\delimsize]}
      #1
   }

\newcommand{\MIS}{\textsc{MIS}\xspace}
\newcommand{\facet}{\textsc{Facet}\xspace}
\newcommand{\pareto}{\textsc{Pareto}\xspace}

\makeatletter
\def\myspecialnode#1{
	\tikz@scan@one@point\pgfutil@firstofone($#1$)
	\pgf@xa=\pgf@x%
	\pgf@ya=\pgf@y%
	\pgfmathparse{1/(\pgf@xa*\pgf@xa/(\pgf@xx*\pgf@xx)+\pgf@ya*\pgf@ya/(\pgf@yy*\pgf@yy))}
	\let\myval=\pgfmathresult
}
\makeatother

\newcommand{\subobj}{z}

	\tikzset{
		highlight/.style={
			fill=accentuating green,
			fill opacity=.2},
		gridnode/.style={
			circle,
			minimum size=18pt,
			inner sep=0pt,
			text=black,
			draw=black},
		extpoint/.style={
			fill,
			circle,
			minimum width=4pt,
			inner sep=0},
	}

\usepackage[ngerman,english]{cleveref}

\crefname{enumi}{}{}
\crefname{equation}{}{}
\creflabelformat{enumi}{#2#1#3}

\theoremstyle{plain}
\newtheorem{theorem}{Theorem}%
\crefname{theorem}{Theorem}{Theorems}

\crefname{proposition}{Proposition}{Propositions}
\newtheorem{lemma}[theorem]{Lemma}%
\crefname{lemma}{Lemma}{Lemmas}
\newtheorem{corollary}[theorem]{Corollary}
\crefname{corollary}{Corollary}{Corollaries}

\theoremstyle{definition}
\newtheorem{remark}[theorem]{Remark}%
\crefname{remark}{Remark}{Remarks}
\newtheorem{definition}[theorem]{Definition}%
\crefname{definition}{Definition}{Definitions}

\crefname{example}{Example}{Examples}

\crefname{table}{Table}{Tables}
\crefname{section}{Section}{Sections}

\usepackage{thmtools}
\makeatletter
\declaretheoremstyle[
qed={},
headpunct={},
]{examplec}
\declaretheorem[numberlike=theorem,style=examplec, name=Example]{examplec}%
\crefname{examplec}{Example}{Example}
\makeatother
\renewcommand\thmcontinues[1]{%
	\ifcsname hyperref\endcsname
		\hyperref[#1]{continued}%
	\else
		continued%
	\fi}

\hypersetup{
 pdfauthor={René Brandenberg, Paul Stursberg},
 pdftitle={Cut Selection for Benders Decomposition},
 pdfsubject={},
 pdfkeywords={},
 colorlinks=true, 
}

\author{René Brandenberg}
\address{Technical University of Munich, Germany}
\email{rene.brandenberg@tum.de}

\author{Paul Stursberg}
\thanks{Paul Stursberg acknowledges funding from Deutsche Forschungsgemeinschaft (DFG) through TUM International Graduate School of Science and Engineering (IGSSE), GSC 81, and by the German Federal Ministry for Economic Affairs and Energy (FKZ 03ET4029) on the basis of a decision by the German Bundestag}
\address{Technical University of Munich, Germany}
\email{paul.stursberg@tum.de}

\title{Cut Selection for Benders Decomposition}

\pgfplotsset{compat=1.14}

\begin{document}

\begin{abstract}
	In this paper, we present a new perspective on cut generation in the context of Benders decomposition. The approach, which is based on the relation between the alternative polyhedron and the reverse polar set, helps us to improve established cut selection procedures for Benders cuts, like the one suggested by \textcite{Fischetti:2010}. Our modified version of that criterion produces cuts which are always supporting and, unless in rare special cases, facet-defining.


	We discuss our approach in relation to the state of the art in cut generation for Benders decomposition. In particular, we refer to Pareto-optimality and facet-defining cuts and observe that each of these criteria can be matched to a particular subset of parameterizations for our cut generation framework. As a consequence, our framework covers the method to generate facet-defining cuts proposed by \textcite{Conforti:2018} as a special case.
\end{abstract}

\maketitle

\section{Introduction}
\label{sec:benders_gen}

Consider a generic optimization problem with two subsets of variables $x$ and $y$ where $x$ is restricted to lie in some set $S \subseteq \reals^n$ and $x$ and $y$ are jointly constrained by a set of $m$ linear inequalities. Such a problem can be written in the following form:
\begin{equation}
	\begin{aligned}
	\min\ &c^\top x + d^\top y\\
	\subjectto 	&Hx+Ay \leq b\\
				&x \in S \subseteq \reals^n\\
				&y \in \reals^k
	\end{aligned}\label{eq:gen_problem}
\end{equation}

The \emph{interaction matrix} $H \in \reals^{m \times n}$
captures the influence of the $x$-variables on the $y$-subproblem: For fixed $x^*$, \eqref{eq:gen_problem} reduces to an ordinary linear program with constraints $Ay \leq b-Hx^*$.

We are interested in cases where the size of the complete problem \cref{eq:gen_problem} leads to infeasibly high computation times (or memory demands), but both the problem over $S$ and the problem resulting from fixing $x$ can separately be solved much more efficiently due to their special structures.
To deal with such problems, \textcite{Benders:1962} introduced a method that works by iterating between these two \enquote{easier} problems:

For a problem of the form \cref{eq:gen_problem}, let the function $\subobj: \reals^n \rightarrow \reals \cup \Set{\pm\infty}$ represent the value of the optimal $y$-part of the objective function for a fixed vector $x$:
\begin{equation}
	\subobj(x) := \min_{y \in \reals^k} \Set*{d^\top y \given Ay \leq b - Hx}\label{eq:gen_eta}
\end{equation}

The corresponding epigraph of $\subobj$ is
\begin{equation}
	\epi(\subobj) = \Set*{(x,\eta) \in \reals^n \times \reals \given \exists y \in \reals^k:
				\begin{gathered}
					Ay \leq b - Hx\\
					d^\top y \leq \eta
				\end{gathered}}.\label{eq:gen_epi}
\end{equation}

Writing $\epi_S(\subobj) := \epi(\subobj) \cap (S \times \reals)$, this provides us with an alternative representation of the optimization problem \cref{eq:gen_problem}:
\begin{equation*}
	\min\Set{c^\top x + \eta \given (x,\eta) \in \epi_S(\subobj)}
\end{equation*}

This representation suggests the following iterative algorithm:
Start by finding a solution $(x^*,\eta^*) \in S \times \reals$ that minimizes $c^\top x + \eta$ without any additional constraints (adding a generous lower bound for $\eta$ to make the problem bounded). If $(x^*,\eta^*) \in \epi(\subobj)$, then $(x^*,\eta^*) \in \epi_S(\subobj)$ (since $x^* \in S$) and the solution is optimal. Otherwise, we add constraints violated by $(x^*,\eta^*)$ but satisfied by all $(x',\eta') \in \epi(\subobj)$ and iterate.
This is of course just an ordinary cutting plane algorithm and the crucial question is how to select a \emph{separating inequality} in each iteration.

The original Benders algorithm uses \emph{feasibility cuts} (cuts with coefficient $0$ for the variable $\eta$) and \emph{optimality cuts} (cuts with non-zero coefficient for the variable $\eta$), depending on whether or not the subproblem that results from fixing the $x$-variables is feasible (see, \eg, \cite{Vanderbeck:2010}). \Textcite{Fischetti:2010}, on the other hand, present a unified perspective that covers both cases: They begin by observing that the subproblem can be seen as a pure feasibility problem, represented by the set
\begin{equation}
	\Set*{y \in \reals^k  \given  \begin{gathered}
									Ay \leq b-Hx^*\\
									d^\top y \leq \eta^*
								\end{gathered}}.\label{eq:gen_subprob_feas}
\end{equation}

This polyhedron will be empty if and only if $(x^*,\eta^*) \notin \epi(\subobj)$ 
and any Farkas certificate for emptiness of \cref{eq:gen_subprob_feas} can be used to derive an additional valid inequality. The set of such certificates (up to positive scaling)
\begin{equation}
	P(x^*,\eta^*) := \Set*{\gamma,\gamma_0 \geq 0 \given \begin{gathered}
																	\gamma^\top A + \gamma_0 d^\top = 0\\
																	\gamma^\top(b-Hx^*)+\gamma_0\eta^*=-1
															\end{gathered}} \label{eq:P_alt}
\end{equation}
is called \emph{alternative polyhedron}.
Thus $P(x^*,\eta^*) = \emptyset$ if and only if $(x^*,\eta^*) \in \epi(\subobj)$ and every point $(\gamma, \gamma_0) \in P(x^*,\eta^*)$ induces an inequality $\gamma^\top(b-Hx)+\gamma_0\eta \geq 0$ that is valid for $\epi(\subobj)$ but violated by $(x^*,\eta^*)$.

This characterization is very useful and has been demonstrated empirically to work well in \cite{Fischetti:2010}. However, it exposes some fundamental issues, which are demonstrated by the following example.

\begin{examplec}\label{ex:alt_poly_rev_polar}
	Consider the following optimization problem:
	\begin{equation}
		\min\ \begin{aligned}[t] x + y\quad&\\
			2x+y &\geq 5\\
			\frac{1}{2}x+y &\geq 3\\
			4x+4y &\geq 14
	\end{aligned}\label{eq:ex_alt_poly_rev_polar}
	\end{equation}

	\begin{figure}
		\centering
		\begin{tikzpicture}
		\begin{axis}[
		width=.6\textwidth,
		xlabel=x,
		ylabel=y,
		axis x line=bottom,
		axis y line=left,
		xmin=0,
		ymin=0
		]

		\addplot[] {5-2*x};
		\addplot[] {3-1/2*x};
		\addplot[] {7/2-x};

		\fill[highlight] (axis cs:0,5) -- (axis cs:4/3,7/3) -- (axis cs:5,0.5) -- (axis cs:5,6) -- cycle;

		\end{axis}
		\end{tikzpicture}
		\caption{Constraints and feasible region for the optimization problem from \cref{ex:alt_poly_rev_polar}.}
		\label{fig:ex_alt_poly_rev_polar}
	\end{figure}
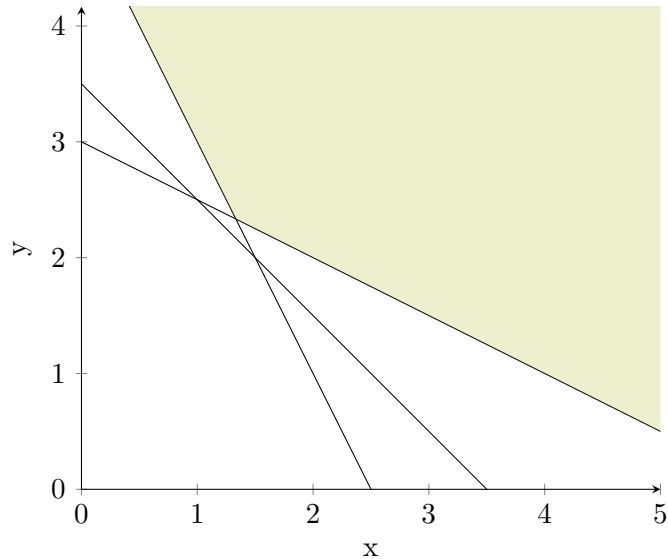

	Note that the constraint $4x+4y \geq 14$ is redundant and does not support the feasible region. Suppose that we want to decompose the problem into its $x$-part and its $y$-part.
%
%
%
%
%
%
%

	Writing the components of $(\gamma,\gamma_0)$ in order $(\gamma_1,\gamma_2,\gamma_3, \gamma_0)$,
  the three extremal points of $P(0,0)$ are
	\begin{align*}
		P_1&=\left(\frac{1}{5},0,0, \frac{1}{5}\right)\\
		P_2&=\left(0,\frac{1}{3},0, \frac{1}{3}\right)\\
		P_3&=\left(0,0,\frac{1}{14}, \frac{2}{7}\right).
	\end{align*}

	As \textcite{Gleeson:1990} showed, each of these points corresponds to a minimal infeasible subsystem fo \cref{eq:ex_alt_poly_rev_polar} with the objective function written in inequality form $x+y \le 0$. Consequently, each extremal point yields one of the original inequalities as a cut. This notably includes the redundant inequality $4x+4y \geq 14$, which does not support the feasible region but is derived from the extremal point $P_3$ in the alternative polyhedron.
\end{examplec}

A cut generated from a point in the alternative polyhedron may thus be very weak, not even supporting the set $\epi(\subobj)$. This is true even if we use a vertex of the alternative polyhedron and even if that vertex minimizes a given linear objective such as the vector $\mathbbm{1}$ as suggested in \cite{Fischetti:2010}.

In the following, we present an improved approach for cut generation in the context of Benders decomposition. Our method can be parametrized by the selection of an objective vector in primal space and produces facet cuts without any additional computational effort for all but a sub-dimensional set of parametrizations. In addition, our method is more robust with respect to the formulation of the problem than the original approach from \cite{Fischetti:2010}. In particular it always generates supporting Benders cuts, avoiding the problem pointed out in the context of \cref{ex:alt_poly_rev_polar} above.

Our method is based on the relation between the alternative polyhedron as introduced above, which is commonly used in the context of Benders cut generation, and the reverse polar set, originally introduced by \textcite{Balas:1964} in the context of transportation problems.

We show that the alternative polyhedron can be viewed as an extended formulation of the reverse polar set, providing us with a parametrizable method to generate cuts with different well-known desirable properties, most notably facet-defining cuts. As a special case, we obtain an (arguably simpler) alternative proof for the method to generate facet-defining cuts proposed by \textcite{Conforti:2018}, if applied to Benders Decomposition. Our works links their approach more directly to previous work on cut selection, both within Benders decomposition (\eg \cite{Fischetti:2010}) and more generally for separation from convex sets (\eg \cite{Cornuejols:2006}).

Before we proceed by investigating different representations of the set of possible Benders cuts, it is useful
to record a general characterization of the set of
normal vectors for cuts separating a point from $\epi(\subobj)$ as defined in \cref{eq:gen_epi}. In the following, we say that a halfspace $\halfspace{\pi}{\alpha} := \Set{x \in \reals^n \given \pi^\top x \leq \alpha}$ is \emph{$x$-separating} for a convex set $C \subset \reals^n$ and a point $x \in \reals^n \setminus C$ if
$x \notin \halfspace{\pi}{\alpha} \supset C$.

\begin{theorem}\label{thm:gen_feascuts}
	Let $\subobj$ be defined as in \cref{eq:gen_eta} such that $\epi(\subobj) \neq \emptyset$ and let $(x^*,\eta^*),(\pi, \pi_0) \in \reals^n \times \reals$.
	Then $(\pi,\pi_0)$ is the normal vector of a $(x^*,\eta^*)$-separating halfspace for $\epi(\subobj)$
  if and only if there exists a vector $\gamma \in \posreals^m$ satisfying
	\begin{align}
		(\pi^\top,\pi_0) \colvec{x^*\\\eta^*} - \gamma^\top b &> 0\label{eq:gen_feascuts_1}\\
		\gamma^\top A - \pi_0 d^\top &= 0\label{eq:gen_feascuts_2}\\
		\gamma^\top H &= \pi^\top\\
		 \pi_0 &\leq 0.\label{eq:gen_feascuts_last}
	\end{align}
\end{theorem}

\begin{proof}
	Let $h_{\epi(\subobj)}(\pi,\pi_0) := \sup\Set{\pi^\top x + \pi_0 \eta \given (x,\eta) \in \epi(\subobj)}$ be the support function of $\epi(\subobj)$ evaluated at $(\pi,\pi_0)$.
	The vector $(\pi,\pi_0)$ is the normal vector of a $(x^*,\eta^*)$-separating halfspace for $\epi(\subobj)$ if and only if

  \begin{equation}
    0 <  (\pi^\top,\pi_0) \colvec{x^*\\\eta^*} - h_{\epi(\subobj)}(\pi, \pi_0).\label{eq:gen_separation_criterion}
  \end{equation}
	By the definition of $\epi(\subobj)$ (which is closed and polyhedral) and then by strong LP duality, we obtain
	\begin{align}
		h_{\epi(\subobj)}(\pi, \pi_0) &= \max_{\substack{x \in \reals^n,y \in \reals^k\\\eta \in \reals}}\Set*{ (\pi^\top,\pi_0) \colvec{x\\\eta}  \given  \begin{aligned}
				Ay &\leq b-Hx\\
				d^\top y &\leq \eta
			\end{aligned}}\label{eq:gen_feascuts_primal}\\
			&= \min_{\substack{\gamma_0 \in \posreals\\\gamma \in \posreals^m}} \Set*{ \gamma^\top b  \given  \begin{aligned}
				\gamma^\top A + \gamma_0 d^\top &= 0\\
				\gamma^\top H &= \pi^\top\\
				-\gamma_0 &= \pi_0
			\end{aligned} }\label{eq:gen_feascuts_dual}.
	\end{align}
	Note that in order for the equality $-\gamma_0=\pi_0$ to hold and \cref{eq:gen_feascuts_dual} to be feasible (and hence \cref{eq:gen_feascuts_primal} to be bounded), we need that $\pi_0 \leq 0$.
  Thus the optimality of any pair $(\gamma,\gamma_0)$ for \eqref{eq:gen_feascuts_dual} is equivalent to
  to the fulfillment of conditions \crefrange{eq:gen_feascuts_1}{eq:gen_feascuts_last}.
\end{proof}

As one can see from the proof above, any $\gamma$ satisfying \crefrange{eq:gen_feascuts_2}{eq:gen_feascuts_last} is an upper bound for 
$h_{\epi(\subobj)}$
. This means that given a certificate $\gamma$ to prove that $(\pi,\pi_0)$ is a normal vector of an $(x^*,\eta^*)$-separating halfspace $\halfspace{(\pi,\pi_0)}{\alpha}$, we immediately obtain a corresponding right hand side $\alpha := \gamma^\top b$.
Furthermore, the definition of the support function $h_{\epi(\subobj)}$ immediately tells us when this right-hand side is actually optimal and the resulting halfspace supports $\epi(\subobj)$:

\begin{remark}\label{thm:gen_supports_Q}
	Let $(x^*,\eta^*) \in \reals^n \times \reals$ and let $(\pi,\pi_0)$ be the normal vector of an $(x^*,\eta^*)$-separating halfspace for $\epi(\subobj)$. If $\gamma$ minimizes $\gamma^\top b$ among all possible certificates in \cref{thm:gen_feascuts}, then the halfspace $\halfspace{(\pi,\pi_0)}{ \gamma^\top b}$ supports the set $\epi(\subobj)$.
\end{remark}
\section{Benders Cuts from the Reverse Polar Set}

While it would be sufficient for the approach from \cite{Fischetti:2010} to obtain an \emph{arbitrary} $(x^*,\eta^*)$-separating halfspace whenever the set in \cref{eq:gen_subprob_feas} is empty, the alternative polyhedron $P(x^*,\eta^*)$ actually completely characterizes the set of \emph{all} possible normal vectors of such halfspaces:

\begin{corollary}\label{thm:alt_poly_all_cuts}
	The alternative polyhedron \cref{eq:P_alt} completely characterizes all normal vectors of $(x^*,\eta^*)$-separating halfspaces for $\epi(\subobj)$. In particular:
	\begin{enumerate}
		\item Let $(\gamma,\gamma_0) \in P(x^*,\eta^*)$. Then $\gamma^\top Hx -\gamma_0 \eta \leq \gamma^\top b$ is violated by $(x^*,\eta^*)$, but satisfied by all $(x,\eta) \in \epi(\subobj)$.
		\item Let $(\pi,\pi_0)$ be the normal vector of a $(x^*,\eta^*)$-separating halfspace for $\epi(\subobj)$. Then there exist $(\gamma,\gamma_0) \in P(x^*,\eta^*)$ and $\lambda \geq 0$ such that $(\gamma^\top H, -\gamma_0) = \lambda \cdot (\pi,\pi_0)$.
	\end{enumerate}
\end{corollary}

Observe, however, that in contrast to \cref{thm:gen_supports_Q}, \cref{thm:alt_poly_all_cuts} does not guarantee that the cut generated from a point in the alternative polyhedron is supporting: A given vector $(\gamma,\gamma_0) \in P(x^*,\eta^*)$ might not minimize $\gamma^\top b$ among all points in $P(x^*,\eta^*)$ which lead to the same cut normal. Indeed, this is precisely what we observed in \cref{ex:alt_poly_rev_polar}, which shows that even a cut generated from a vertex of the alternative polyhedron might not be supporting.

Alternatively, as argued by \textcite{Cornuejols:2006}, we can characterize the set of
normals of separating cuts by the \emph{reverse polar set} of $\epi(\subobj)-(x^*,\eta^*)$, which is introduced in \textcite{Balas:1964} and defined as follows:

\begin{definition}
	Let $C \subseteq \reals^n$ be a convex set. Then the \emph{reverse polar set} $C^-$ of $C$ is defined
  as
	\begin{equation*}
		C^- := \Set*{c \in \reals^n \given c^\top x \leq -1 \text{ forall } x \in C}.
	\end{equation*}
	It is thus a subset of the \emph{polar cone}
	\begin{equation*}
		\pos(C)^\circ := \Set*{c \in \reals^n \given c^\top x \leq 0 \text{ forall } x \in C}.
	\end{equation*}

\end{definition}

 Note that by this definition, the reverse polar set is given by a, possibly infinite, intersection of halfspaces (an \emph{$\CH$-representation})
. If $C$ is a polyhedron, it is actually sufficient to consider halfspaces corresponding to vertices of $C$.  Nonetheless, even for a polyhedron $C$,
efficiently computing an $\CH$-representation of $C^-$ may not be possible in general
as computing all the vertices of $C$, when given in $\CH$-representation, is NP-hard \cite{Khachiyan:2008}).
Given an explicit $\CH$-representation of $C$, there does exist however an extended formulation for $C^-$ based on the coefficients from a convex combination of the \emph{vertices} of $C^-$, which can be easily obtained.

Even without an explicit $\CH$-representation of the set $\epi(z)$ (which is itself known to us only by its extended formulation \cref{eq:gen_epi}), we can use \cref{thm:gen_feascuts} and the fact that the set defined by the inequalities \crefrange{eq:gen_feascuts_1}{eq:gen_feascuts_last} is homogenous, to obtain the following extended formulation of the reverse polar set:
	\begin{equation}
		(\epi(\subobj)-(x^*,\eta^*))^- =
		\Set*{(\pi,\pi_0) \in \reals^n\times\reals \given \exists \gamma \geq 0:
						\begin{aligned}
							(\pi^\top,\pi_0) \colvec{x^*\\\eta^*} - \gamma^\top b &\geq 1\\
							A^\top \gamma - \pi_0 d &= 0\\
							H^\top \gamma &= \pi\\
							\pi_0 &\leq 0
						\end{aligned}}.\label{eq:gen_rev_polar}
	\end{equation}

Note furthermore that, as a consequence of \cref{thm:gen_supports_Q}, we can compute for any given normal vector $(\pi,\pi_0)$ a supporting inequality (if one exits) by solving problems \cref{eq:gen_feascuts_primal} or \cref{eq:gen_feascuts_dual}.

	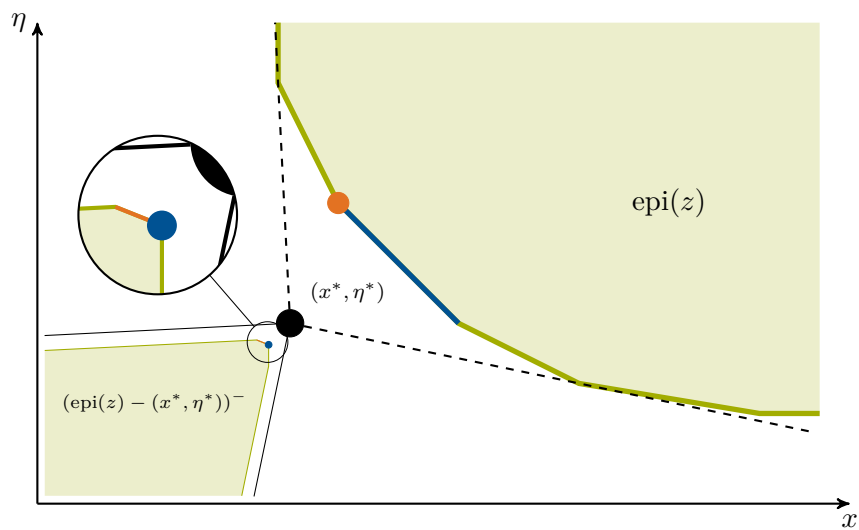
\begin{figure}
		\centering
		\begin{tikzpicture}[spy using outlines={circle, magnification=3.9, size=60, connect spies}, scale=0.8]
			\draw[thick, ->, >=stealth'] (-3,-1) -- (-3,7) node[left] {$\eta$};
			\draw[thick, ->, >=stealth'] (-3,-1) -- (10.5,-1) node[below] {$x$};

			\fill[highlight] (1,7) -- (1,6) -- (2,4) -- (4,2) -- (6,1) -- (9,0.5) -- (10,0.5) -- (10,7) -- cycle;
			\draw[line width=2pt, accentuating green] (1,7) -- (1,6) -- (2,4) -- (4,2) -- (6,1) -- (9,0.5) -- (10,0.5);

			\node[draw, fill, circle, label=above right:{\scriptsize $(x^*,\eta^*)$}] (sol) at (1.2,2) {};
			\node at (7.5,4) {$\epi(\subobj)$};

			\draw[thick, dashed] (sol) -- ($(sol)!1.25!(1,6)$);
			\draw[thick, dashed] (sol) -- ($(sol)!1.8!(6,1)$);
			\draw[] (sol) -- ($(sol)!1.02!90:(1,6)$);
			\draw[] (sol) -- ($(sol)!0.6!-90:(6,1)$);

			\draw[accentuating green]
					\pgfextra{\myspecialnode{(sol)-(1,6)!(sol)!(2,4)}}
					($($(sol)!-\myval!($(1,6)!(sol)!(2,4)$)$)!-6.2!($(sol)!($(sol)!-\myval!($(1,6)!(sol)!(2,4)$)$)!($(sol)!0.2!(1,6)$)$)$)
																		--	($(sol)!-\myval!($(1,6)!(sol)!(2,4)$)$)
					\pgfextra{\myspecialnode{(sol)-(2,4)!(sol)!(4,2)}}	--	($(sol)!-\myval!($(2,4)!(sol)!(4,2)$)$)
					\pgfextra{\myspecialnode{(sol)-(4,2)!(sol)!(6,1)}}	--	($(sol)!-\myval!($(4,2)!(sol)!(6,1)$)$)
						-- ($($(sol)!-\myval!($(4,2)!(sol)!(6,1)$)$)!-2.85!($(sol)!($(sol)!-\myval!($(4,2)!(sol)!(6,1)$)$)!($(sol)!0.2!(6,1)$)$)$);

			\fill[highlight]
				\pgfextra{\myspecialnode{(sol)-(1,6)!(sol)!(2,4)}}
				($($(sol)!-\myval!($(1,6)!(sol)!(2,4)$)$)!-6.2!($(sol)!($(sol)!-\myval!($(1,6)!(sol)!(2,4)$)$)!($(sol)!0.2!(1,6)$)$)$)
																	--	($(sol)!-\myval!($(1,6)!(sol)!(2,4)$)$)
				\pgfextra{\myspecialnode{(sol)-(2,4)!(sol)!(4,2)}}	--	($(sol)!-\myval!($(2,4)!(sol)!(4,2)$)$)
				\pgfextra{\myspecialnode{(sol)-(4,2)!(sol)!(6,1)}}	--	($(sol)!-\myval!($(4,2)!(sol)!(6,1)$)$)
					-- ($($(sol)!-\myval!($(4,2)!(sol)!(6,1)$)$)!-2.85!($(sol)!($(sol)!-\myval!($(4,2)!(sol)!(6,1)$)$)!($(sol)!0.2!(6,1)$)$)$)
				\pgfextra{\myspecialnode{(sol)-(1,6)!(sol)!(2,4)}}	-|
				($($(sol)!-\myval!($(1,6)!(sol)!(2,4)$)$)!-6.2!($(sol)!($(sol)!-\myval!($(1,6)!(sol)!(2,4)$)$)!($(sol)!0.2!(1,6)$)$)$);

			\node[font=\tiny] at ($(sol)+(-2.25,-1.3)$) {$(\epi(\subobj) - (x^*,\eta^*))^-$};
			\spy on ($(sol)- (0.3,0.25)$) in node at (-1,3.8);

			\draw[line width=2pt, tumblue] (2,4) -- (4,2);
			\draw[accentuating orange]
					\pgfextra{\myspecialnode{(sol)-(1,6)!(sol)!(2,4)}}		($(sol)!-\myval!($(1,6)!(sol)!(2,4)$)$)
					\pgfextra{\myspecialnode{(sol)-(2,4)!(sol)!(4,2)}}	--	($(sol)!-\myval!($(2,4)!(sol)!(4,2)$)$);
			\node[gridnode, minimum size=8pt, accentuating orange, fill] at (2,4) {};
			\draw \pgfextra{\myspecialnode{(sol)-(2,4)!(sol)!(4,2)}} node[gridnode, minimum size=2.5pt, tumblue, fill] at ($(sol)!-\myval!($(2,4)!(sol)!(4,2)$)$) {};
		\end{tikzpicture}
		\caption{The reverse polar set $(\epi(\subobj) - (x^*,\eta^*))^-$ and the corresponding polar cone (drawn in a coordinate system with $(x^*,\eta^*)$ as the origin). It can be seen that $(\epi(\subobj) - (x^*,\eta^*))^-$ is contained in the polar cone $\pos(\epi(\subobj) - (x^*,\eta^*))^\circ$ (indicated by the black solid lines) but offers a \enquote{richer} boundary from which we can choose cut normals. Specifically, for each vertex $v$ of $(\epi(\subobj) - (x^*,\eta^*))^-$ there exists a facet of $\epi(\subobj)$ with normal vector $v$ and vice versa (see \cref{thm:rev_polar_facets}).}
		\label{fig:gen_reverse_polar}
	\end{figure}

We thus have at our disposal two alternative characterizations of the set of possible normal vectors of $(x^*,\eta^*)$-separating halfspaces: The alternative polyhedron and the reverse polar set. Despite their similarity, subtle differences exist between both representations that affect their usefulness for the generation of Benders cuts.

It should be noted at this point that we are not the first ones to notice the similarity between the approaches of  \textcite{Cornuejols:2006} and \textcite{Fischetti:2010}. Indeed, the work of \textcite{Cornuejols:2006} is explicitly cited in \cite{Fischetti:2010}, albeit only in a remark about the possibility to exchange normalization and objective function in optimization problems over the alternative polyhedron (see \cref{thm:gen_alt_rep_alt_poly} below).

%

Before we proceed, we introduce a variant of the alternative polyhedron, the \emph{relaxed alternative polyhedron}, which is also used in \cite{Gleeson:1990}. We will see that it is equivalent to the original alternative polyhedron for almost all purposes, but can more easily be connected to the reverse polar set:

\begin{definition}\label{def:relaxed_alt_polyhedron}
	Let a problem of the form \cref{eq:gen_problem} and a point $(x^*, \eta^*) \in \reals^n \times \reals$ be given. The \emph{relaxed alternative polyhedron} $P^\leq(x^*,\eta^*)$ is defined as
	\begin{equation*}
		P^\leq(x^*,\eta^*) := \Set*{\gamma,\gamma_0 \geq 0 \given \begin{aligned}
										\gamma^\top A + \gamma_0 d^\top &= 0\\
										\gamma^\top(b-Hx^*)+\gamma_0\eta^* &\leq -1
										\end{aligned}}.
	\end{equation*}
\end{definition}

To motivate the above definition, observe that optimization problems over the original and the relaxed alternative polyhedron are equivalent, provided that the optimization problem over the relaxed alternative polyhedron has a finite non-zero optimum:

\begin{remark}\label{thm:alt_poly_relaxed_original}
	Let $\subobj$ be defined as in \cref{eq:gen_eta} and let $(x^*,\eta^*) \in \reals^n \times \reals$. Let $(\tilde\omega, \tilde\omega_0) \in \reals^m \times \reals$ be such that $\max\Set{\tilde\omega^\top \gamma + \tilde\omega_0\gamma_0 \given \gamma,\gamma_0 \in P^\leq(x^*,\eta^*)} < 0$.
	Then the sets of optimal solutions for $\tilde\omega^\top \gamma + \tilde\omega_0\gamma_0$ over $P^\leq(x^*,\eta^*)$ and $P(x^*,\eta^*)$ are identical.
	Furthermore, every vertex of $P^\leq(x^*,\eta^*)$ is also a vertex of $P(x^*,\eta^*)$.
\end{remark}

The following key theorem now becomes a trivial observation. However, to our knowledge, the relation between the alternative polyhedron and the reverse polar set has not been made explicit in a similar fashion before.

\begin{theorem}\label{lem:P_alt_reverse_polar}
	Let $\subobj$ be defined as in \cref{eq:gen_eta} and $(x^*, \eta^*) \in \reals^n \times \reals$. Then
	\begin{equation*}
		(\epi(\subobj)-(x^*,\eta^*))^- = \begin{pmatrix} H^\top & 0 \\ 0 & -1 \end{pmatrix} \cdot P^\leq(x^*,\eta^*).
	\end{equation*}
\end{theorem}

\subsection{Cut-Generating Linear Programs}
\label{sec:gen_cut_gen_opt_probs}

One way to select a particular cut normal from the reverse polar set or the alternative polyhedron is by maximizing a linear objective function over these sets. Using \cref{lem:P_alt_reverse_polar}, we can derive the precise relation between optimization problems over the reverse polar set and the alternative polyhedron.

\begin{corollary}\label{thm:dual_problem_equiv}
	Let $\subobj$ be defined as in \cref{eq:gen_eta}, $(x^*,\eta^*), (\omega,\omega_0) \in \reals^n \times \reals$ and
\begin{equation}
	(\tilde \omega, \tilde \omega_0) := (H\omega, -\omega_0).\label{eq:dual_problem_equiv_obj}
\end{equation}
	Then $(\pi,\pi_0)$ is an optimal solution to the problem
\begin{equation}
	\max \Set*{\omega^\top \pi + \omega_0 \pi_0 \given (\pi, \pi_0) \in (\epi(\subobj)-(x^*,\eta^*))^-}\label{eq:dual_problem_equiv_rev}
\end{equation}
if and only if there exists $\gamma^*$ such that $H^\top\gamma^* = \pi$ and $(\gamma^*,-\pi_0)$ is an optimal solution to the problem
\begin{equation}
	\max \Set*{\tilde \omega^\top \gamma + \tilde \omega_0 \gamma_0 \given (\gamma, \gamma_0) \in P^\leq(x^*,\eta^*)}.\label{eq:dual_problem_equiv_alt}
\end{equation}
Furthermore, the objective values of both optimization problems are identical.
\end{corollary}

\begin{proof}
  Let $(\pi,\pi_0)$ be an optimal solution to \cref{eq:dual_problem_equiv_rev}. By \cref{lem:P_alt_reverse_polar}, there exists a vector $\gamma^*$ with $H^\top\gamma = \pi$ such that $(\gamma, -\pi_0) \in P^\leq(x^*,\eta^*)$.
  Now, let $(\gamma', \gamma'_0)$ be an arbitrary point in $P^\leq(x^*,\eta^*)$. By \cref{lem:P_alt_reverse_polar},
  $(H^\top\gamma',-\gamma'_0) \in (\epi(\subobj)-(x^*,\eta^*))^-$ and thus from the optimality of  $(\pi,\pi_0)$ for  \cref{eq:dual_problem_equiv_rev} we obtain
  \begin{multline*}
    \tilde \omega^\top \gamma' + \tilde \omega_0 \gamma'_0 = (H\omega)^\top\gamma'-\omega_0 \gamma'_0 = \omega^\top (H^\top\gamma') + \omega_0 (-\gamma'_0)  \\
    \le \omega^\top \pi + \omega_0 \pi_0 =   (H\omega)^\top \gamma^* - \omega_0 (-\pi_0) = \tilde \omega^\top \gamma^* + \tilde \omega_0 (-\pi_0),
  \end{multline*}
  which proves the optimality of $(\gamma^*,-\pi_0)$ for \cref{eq:dual_problem_equiv_alt}.

  Similarly, let $(\gamma,\gamma_0)$ be an optimal solution to \cref{eq:dual_problem_equiv_alt}, $\pi := H^\top \gamma$, and $\pi_0 := -\gamma_0$, which means by \cref{lem:P_alt_reverse_polar}, $(\pi,\pi_0) \in (\epi(\subobj)-(x^*,\eta^*))^-$.

  Now, let $(\pi',\pi'_0)$ be an arbitrary point in $(\epi(\subobj)-(x^*,\eta^*))^-$. By \cref{lem:P_alt_reverse_polar}, there exists $\gamma'$ with $H^\top \gamma' = \pi'$ such that $(\gamma',-\pi'_0) \in P^\leq(x^*,\eta^*)$. Using the optimality of $(\gamma,\gamma_0)$ for \cref{eq:dual_problem_equiv_alt}, we obtain
  \begin{multline*}
    \omega^\top \pi' + \omega_0 \pi'_0 =  (H\omega)^\top\gamma' - \omega_0 (-\pi'_0) = \tilde \omega^\top \gamma' + \tilde \omega_0 (-\pi'_0)  \\
    \le  \tilde\omega^\top \gamma + \tilde\omega_0 \gamma_0  = \omega^\top H^\top \gamma - \omega_0 \gamma_0 =  \omega^\top \pi + \omega_0 \pi_0,
  \end{multline*}
  which proves the optimality of $(\pi,\pi_0)$ for \cref{eq:dual_problem_equiv_rev}.
\end{proof}

Depending on the particular application, the structure of the matrix $H$ can vary in many ways, but in line with our assumption that the \emph{master} problem should be significantly smaller than the \emph{subproblem}, it is reasonable to assume that $H$ has many more rows than columns.

In this sense, we will henceforth use the relaxed alternative polyhedron as an \emph{extended formulation} for the reverse polar set, which in particular is always polynomial in size. This will allow us to generate Benders cuts from points in the reverse polar set while algorithmically relying on the relaxed alternative polyhedron, an explicit description of which is generally trivial to obtain.

Note that the optimization problem stated in \cref{eq:dual_problem_equiv_alt} is technically more general, since there is no reason to limit ourselves to objective functions of the form \cref{eq:dual_problem_equiv_obj} \emph{a priori}. If we choose a different objective function, we still obtain a valid cut. However, since there may be no objective function $(\omega, \omega_0)$ such that the resulting cut normal is optimal for \cref{eq:dual_problem_equiv_rev}, we lose some of the properties associated with optimal solutions from the reverse polar set.

Indeed, this is the approach that \textcite{Fischetti:2010} take: They use the problem in \cref{eq:dual_problem_equiv_alt} with $\tilde \omega_m = 0$ for all $m$ which correspond to rows of zeros in the interaction matrix $H$ and $\tilde \omega_m = 1$ for all other $m$, as well as $\tilde\omega_0 = 1$ (or some other manual scaling factor). In general, there exists no vector $(\omega,\omega_0)$ such that this choice can be obtained by \cref{eq:dual_problem_equiv_obj}.

We now take a closer look at the role of objective functions in the context of \cref{ex:alt_poly_rev_polar}:

\begin{examplec}[continues={ex:alt_poly_rev_polar}]
	In the situation of the optimization problem \cref{eq:ex_alt_poly_rev_polar}, observe that the point $P_3$ actually minimizes the 1-norm over $P(0,0)$ and is hence the unique result of the (unscaled) selection procedure from \cite{Fischetti:2010}. On the other hand the transformation from \cref{lem:P_alt_reverse_polar} actually maps this point, which lead to a non-supporting cut, to the interior of the reverse polar set. It will therefore never appear as an optimal solution of any linear optimization problem.

  In order to obtain a supporting cut, we only have to make sure that the objective
  that we use can be written in the form $(H\omega, -\omega_0)$. In our example, if we choose the objective function over the alternative polyhedron from the set
	\begin{equation*}
		\Set*{ \left(\colvec{-2\\-\nicefrac{1}{2}\\-4} \cdot \omega,-\omega_0\right) \given \omega, \omega_0 \in \reals},
	\end{equation*}
	then the point $P_3 \in P(0,0)$ is never optimal. 
\end{examplec}

One interesting difference between the alternative polyhedron and the reverse polar set, which can be verified using the above example, is their different behavior with respect to algebraic operations on the set of inequalities: If, for instance, we scale one of the inequalities by a positive factor, the reverse polar set remains unchanged (just as the feasible region defined by the set of inequalities). The alternative polyhedron, on the other hand, is distorted in response to the scaling of the system of inequalities. If an objective function is used which does not take this scaling into account, such as the vector of zeros and ones proposed by \textcite{Fischetti:2010}, then the selected cut might change depending on the scaling factor. Even selecting a suitable manual scaling factor $\tilde\omega_0$ as mentioned above cannot fix this, since it cannot scale individual constraints against each other.

Combining our results from this section, we obtain the following statement:

\begin{corollary}\label{thm:gen_opt_gamma}
	Let $\subobj$ be defined as in \cref{eq:gen_eta} and $(x^*,\eta^*), (\omega,\omega_0) \in \reals^n \times \reals$, $(\tilde \omega, \tilde \omega_0) := (H\omega, -\omega_0)$, and $(\gamma,\gamma_0) \in P(x^*,\eta^*)$
	be maximal with respect to the objective $(\tilde \omega, \tilde \omega_0)$ such that $\tilde \omega^\top \gamma + \tilde \omega_0 \gamma_0 < 0$.
	Then the inequality $\gamma^\top H x - \gamma_0 \eta \leq \gamma^\top b$ supports $\epi(\subobj)$.
\end{corollary}

\begin{proof}
	Let $(\pi,\pi_0) := (H^\top \gamma, -\gamma_0)$. Then, by \cref{thm:gen_supports_Q}, the statement is true if $\gamma$ minimizes $\gamma^\top b$ among all possible certificates for the vector $(\pi,\pi_0)$ in \cref{thm:gen_feascuts}. It is easy to verify that $\gamma$ is indeed a valid certificate for $(\pi,\pi_0)$ in \cref{thm:gen_feascuts}. For a contradiction, we hence assume that it does not minimize $\gamma^\top b$. Let thus $\gamma' \geq 0$ be an alternative certificate for $(\pi,\pi_0)$ with $\gamma'^\top b < \gamma^\top b$. Then from \crefrange{eq:gen_feascuts_1}{eq:gen_feascuts_last} we obtain that $\gamma'^\top A-\pi_0 d^\top = 0$ and $\gamma'^\top H = \pi^\top$.
	Furthermore, since $(\gamma, \gamma_0) \in P^\leq(x^*,\eta^*)$,
	\begin{equation*}
		\gamma'^\top(b-Hx^*)+\gamma_0\eta^* = \gamma'^\top b - \pi^\top x^* + \gamma_0\eta^* < \gamma^\top b - \pi^\top x^* + \gamma_0\eta^* \leq -1.
	\end{equation*}

	We can thus scale $(\gamma', \gamma_0)$ by an appropriate factor $\lambda \in (0,1)$ to obtain that $\lambda \cdot (\gamma', \gamma_0) \in P^\leq(x^*,\eta^*)$ and
	\begin{align*}
		\tilde \omega^\top (\lambda \gamma') + \tilde \omega_0 \cdot (\lambda \gamma_0)
		&= \lambda \cdot (\tilde\omega^\top \gamma' + \tilde\omega_0 \gamma_0)
    = \lambda \cdot (\tilde\omega^\top \gamma+ \tilde\omega_0 \gamma_0) > \tilde\omega^\top \gamma + \tilde\omega_0 \gamma_0 .
	\end{align*}

	On the other hand, by \cref{thm:alt_poly_relaxed_original}, if $(\gamma,\gamma_0)$ maximizes the objective $(\tilde \omega, \tilde \omega_0)$ over $P(x^*,\eta^*)$, then it is also maximal within $P^\leq(x^*,\eta^*)$, a contradiction.
	By \cref{thm:gen_supports_Q}, this implies that the inequality $\gamma^\top H x - \gamma_0 \eta \leq \gamma^\top b$ does indeed support $\epi(\subobj)$.
\end{proof}

A critical requirement for \cref{thm:gen_opt_gamma} is that $\tilde \omega^\top \gamma + \tilde \omega_0 \gamma_0 < 0$. \Textcite[Theorem 2.3]{Cornuejols:2006} establish some criteria on the objective function for which optimization problems over the reverse polar set are bounded. We have simplified the notation for our purposes and rephrased the relevant parts of the theorem according to our terminology.

\begin{theorem}
  \label{thm:cornuejols_reverse_gauge}
	Let $(x^*,\eta^*) \notin \epi(\subobj)$, $(\omega,\omega_0) \in \reals^n \times \reals$, and
	\begin{equation*}
		z^* := \max \Set*{\omega^\top \pi + \omega_0 \pi_0 \given (\pi,\pi_0) \in (\epi(\subobj)-(x^*,\eta^*))^-}.
	\end{equation*}
	Then
	\begin{equation*}
		z^* \begin{cases}
			\leq 0 & \text{if } (\omega,\omega_0) \in \cl(\pos(\epi(\subobj)-(x^*,\eta^*)))\\
			= \infty & \text{otherwise.}
		\end{cases}
	\end{equation*}
	Furthermore, if $(\omega,\omega_0) \in (\epi(\subobj)-(x^*,\eta^*))$, then $z^* \leq -1$.
\end{theorem}

Note in particular that the last part of the above statement implies $z^* < 0$ whenever $(\omega,\omega_0) \in \pos(\epi(\subobj)-(x^*,\eta^*)) \setminus \Set{0}$, which provides us with a large variety of objective functions for which $ \omega^\top \gamma + \omega_0 \gamma_0 < 0$ in the optimal solution. By \cref{thm:dual_problem_equiv,thm:gen_opt_gamma}, this means that the cut which results from maximizing these objectives over the reverse polar set is guaranteed to be supporting.

\subsection{Alternative Representations}

Finally, to conclude our dictionary of cut-generating optimization problems, we derive an alternative representation of the optimization problem in \cref{eq:dual_problem_equiv_alt}, which will turn out to be much more useful in practice. For instance, the structure of the resulting problem will be very similar to the original subproblem, which makes it easy to use existing solution algorithms for the subproblem in a cut-generating program.

\Textcite[Theorem 4.2]{Cornuejols:2006} prove that linear optimization problems over the reverse polar set can be evaluated in terms of the support function of the original set (in our case $\epi(\subobj)-(x^*,\eta^*)$). This can also be applied to the alternative polyhedron, as mentioned (without proof) by \textcite{Fischetti:2010}. The following lemma makes a statement similar to \textcite[Theorem 4.2]{Cornuejols:2006}, which is applicable to a wider range of settings. For the proof, we refer to \cite[Theorem 3.20]{Stursberg:2019}.

\begin{lemma}\label{thm:gen_exchange_normalization_cone}
	Let $K \subseteq \reals^n$ be a cone and $c_1, c_2 \in \reals^n$. Consider the optimization problems
	\begin{equation}
		\max \Set*{c_1^\top x \given x \in K, c_2^\top x = -1}\label{eq:gen_exchange_normalization_cone_3}
	\end{equation}
	and
	\begin{equation}
		\max \Set*{c_2^\top x \given x \in K, c_1^\top x \geq 1}.\label{eq:gen_exchange_normalization_cone_4}
	\end{equation}
	Then the following hold:
	\begin{enumerate}
		\item If $x^*$ is an optimal solution for \cref{eq:gen_exchange_normalization_cone_3} with objective value $\xi > 0$, then $\frac{1}{\xi} \cdot x^*$ is an optimal solution for \cref{eq:gen_exchange_normalization_cone_4} with objective value $-\frac{1}{\xi}$.
		\item Conversely, if $x^*$ is an optimal solution for \cref{eq:gen_exchange_normalization_cone_4} with objective value $\xi < 0$, then $-\frac{1}{\xi} \cdot x^*$ is an optimal solution for \cref{eq:gen_exchange_normalization_cone_4} with objective value $-\frac{1}{\xi}$.
	\end{enumerate}
\end{lemma}

This lemma allows us to solve optimization problems of the form \cref{eq:dual_problem_equiv_alt} by instead resorting to the optimization problem
	\begin{align}
		\max_{\gamma, \gamma_0 \geq 0}\ &\gamma^\top (Hx^*-b) - \gamma_0\eta^*\label{eq:gen_exchange_normalization_alt_obj}\\
		&\gamma^\top A + \gamma_0 d^\top = 0\\
		&\tilde\omega^\top \gamma + \tilde\omega_0 \gamma_0 = -1.\label{eq:gen_exchange_normalization_alt_last}
	\end{align}

Let $(\tilde\omega,\tilde\omega_0) \in \reals^m \times \reals$ and let $(\gamma^*,\gamma_0^*)$ denote an optimal solution with value $\xi > 0$ for \crefrange{eq:gen_exchange_normalization_alt_obj}{eq:gen_exchange_normalization_alt_last}. Applying \cref{thm:gen_exchange_normalization_cone} with $c_1 := (Hx^*-b,-\eta^*)$, $c_2 := (\tilde\omega,\tilde\omega_0)$ and $K := \Set{(\gamma,\gamma_0) \geq 0 \given \gamma^\top A + \gamma_0 d^\top = 0}$, we obtain that $\frac{1}{\xi} \cdot (\gamma^*,\gamma_0^*)$ is an optimal solution with value $-\frac{1}{\xi}$ for \cref{eq:dual_problem_equiv_alt}.

%

The structural similarity of \crefrange{eq:gen_exchange_normalization_alt_obj}{eq:gen_exchange_normalization_alt_last} to the original problem becomes more apparent when we consider the dual problem:

\begin{corollary}\label{thm:gen_alt_rep_alt_poly}
	Let $(\tilde\omega,\tilde\omega_0) \in \reals^m \times \reals$ and $(\lambda,x,y)$ be an optimal solution for the problem
	\begin{align}
		\min\ &\lambda\label{eq:gen_alt_rep_alt_poly_obj}\\
			&Ay \leq b - Hx^* - \lambda \, \tilde\omega\\
			&d^\top y  \leq \eta^* -  \lambda \, \tilde\omega_0\label{eq:gen_alt_rep_alt_poly_last}
	\end{align}
	with $\lambda > 0$. Denote the corresponding dual solution by $(\gamma,\gamma_0)$. Then $\frac{1}{\lambda}(\gamma,\gamma_0)$ is an optimal solution for
	\begin{equation*}
		\max \Set*{\tilde\omega^\top \gamma + \tilde\omega_0 \gamma_0 \given (\gamma,\gamma_0) \in P^\leq(x^*,\eta^*)}
	\end{equation*}
	with objective value $-\frac{1}{\lambda}$.
\end{corollary}

	Note that, together with our observations in the context of the definition of the alternative polyhedron \cref{eq:P_alt}, this means in particular that
	 \begin{enumerate}
	 	\item whenever \crefrange{eq:gen_alt_rep_alt_poly_obj}{eq:gen_alt_rep_alt_poly_last} has objective value 0, then the alternative polyhedron is empty and $(x^*,\eta^*) \in \epi(\subobj)$, and
	 	\item whenever \crefrange{eq:gen_alt_rep_alt_poly_obj}{eq:gen_alt_rep_alt_poly_last} is feasible with (finite) objective value greater than 0, then \cref{eq:dual_problem_equiv_rev} and \cref{eq:dual_problem_equiv_alt} have objective values strictly less than 0, which means that the requirements for \cref{thm:alt_poly_relaxed_original,thm:gen_opt_gamma} are satisfied.
	 \end{enumerate}


\begin{remark}\label{rem:gen_alt_rep_rev_polar}
If $(\tilde \omega, \tilde \omega_0) := (H\omega, -\omega_0)$, then the optimization problem \crefrange{eq:gen_alt_rep_alt_poly_obj}{eq:gen_alt_rep_alt_poly_last} becomes
	\begin{align}
		\min\ &\lambda\label{eq:gen_alt_rep_rev_polar1}\\
			&Ay \leq b - H(x^* + \lambda \cdot \omega)\\
			&d^\top y \leq \eta^* +\omega_0 \lambda\label{eq:gen_alt_rep_rev_polar3}
	\end{align}
%

The difference between the formulations from \cref{thm:gen_alt_rep_alt_poly} and \cref{rem:gen_alt_rep_rev_polar} lies in how they relax the original problem: In \crefrange{eq:gen_alt_rep_alt_poly_obj}{eq:gen_alt_rep_alt_poly_last}, the relaxation works on the level of individual inequalities by relaxing their right-hand sides, whereas in \crefrange{eq:gen_alt_rep_rev_polar1}{eq:gen_alt_rep_rev_polar3} it works on the level of the master solution $(x^*,\eta^*)$, allowing us to choose a possibly more advantageous value for the vector $x$ itself.
\end{remark}

\section{Cut Selection}
\label{sec:cut_criteria}

As we have seen in the previous section, Benders' decomposition can be viewed as an instance of a classical cutting plane algorithm (\cref{thm:gen_feascuts}). The Benders subproblem takes the role of the separation problem and the alternative polyhedron that is commonly used to select a Benders cut is a higher-dimensional representation (an \emph{extended formulation}) of the reverse polar set, which characterizes all possible cut normals (\cref{lem:P_alt_reverse_polar}).

Finally, \cref{thm:gen_alt_rep_alt_poly,rem:gen_alt_rep_rev_polar} show that selecting a cut normal by a linear objective over the reverse polar set or the alternative polyhedron can be interpreted as two different relaxations \crefrange{eq:gen_alt_rep_alt_poly_obj}{eq:gen_alt_rep_alt_poly_last} and \crefrange{eq:gen_alt_rep_rev_polar1}{eq:gen_alt_rep_rev_polar3} of the original Benders feasibility subproblem \cref{eq:gen_subprob_feas}. The former relaxation provides more flexibility with respect to the choice of parameters and coincides with the latter for a particular selection of the objective function.

Cut selection is one of four major areas of algorithmic improvements for Benders decomposition that recent work has focused on (see, \eg, the very extensive literature review in \cite{Rahmaniani:2017}).
A number of selection criteria for Benders cuts have previously been explicitly proposed in the literature. Many of them also arise naturally from our discussion and analysis of the Benders decomposition algorithm above. We will first present these criteria in the way they typically appear in the literature and then link them to the reverse polar set and/or the alternative polyhedron.

\subsection{Minimal Infeasible Subsystems}
\label{sec:gen_mis}
The work of \textcite{Fischetti:2010} is based on the premise that \enquote{one is interested in detecting a \enquote{minimal source of infeasibility}} whenever the feasibility subproblem \cref{eq:gen_subprob_feas} is empty. They hence suggest to generate Benders cuts based on Farkas certificates that correspond to \emph{minimal infeasible subsystems (MIS)} of \cref{eq:gen_subprob_feas}. \Textcite{Fischetti:2010} empirically study the performance of \MIS-cuts on a set of multi-commodity network design instances. Their results suggest that \MIS-based cut selection outperforms the standard implementation of Benders decomposition by a factor of at least 2-3. Furthermore, this advantage increases substantially when focusing on harder instances (e.g. those which could not be solved by the standard implementation within 10 hours).

We define this criterion as follows:

\begin{definition}
	Let $\subobj$ be defined as in \cref{eq:gen_eta} and let $(\pi,\pi_0) \in \reals^n \times \reals$. We say that $(\pi,\pi_0)$ \emph{satisfies the \MIS criterion} if there exists $(\gamma, \gamma_0) \geq 0$ such that $\pi=H^\top \gamma, \pi_0=-\gamma_0$ and the inequalities which correspond to the non-zero components of $(\gamma,\gamma_0)$ form a minimal infeasible subsystem of \cref{eq:gen_subprob_feas}.
\end{definition}

Note that we have defined the \MIS criterion as a property of a normal vector, rather than a property of a cut. The reason for this is that the cut normal is the only relevant choice to make, given that an optimal right-hand side for each cut normal is provided by \cref{thm:gen_opt_gamma}. Accordingly, we will call any cut with a normal vector that satisfies the \MIS criterion a \MIS-cut.

\Textcite{Gleeson:1990} show that the set of $(\gamma, \gamma_0)$ that appear in the above definition is exactly (up to homogeneity) the set of vertices of the alternative polyhedron:

\begin{theorem}
	Let $(x^*,\eta^*) \in \reals^n \times \reals$. For each vertex $v$ of the (relaxed) alternative polyhedron \cref{eq:P_alt}, the set of constraints corresponding to the non-zero entries of $v$ forms a minimal infeasible subsystem of \cref{eq:gen_subprob_feas}. Conversely, for every minimal infeasible subsystem, there exists a vertex of the alternative polyhedron.
\end{theorem}

This immediately provides a characterization of cut normals which satisfy \MIS in terms of the alternative polyhedron, which is also used in \cite{Fischetti:2010}. Using \cref{lem:P_alt_reverse_polar}, we can furthermore transfer one direction of the characterization to the reverse polar set:

\begin{corollary}\label{cor:alt_poly_MIS}
	Let $\subobj$ be defined as in \cref{eq:gen_eta} and $(x^*,\eta^*) \in \reals^n \times \reals$. The vector $(\pi,\pi_0)$ satisfies the \MIS criterion if and only if there is an extremal point $(\gamma, \gamma_0)$ of $P(x^*,\eta^*)$ such that $(\pi,\pi_0)=(H^\top \gamma,-\gamma_0)$.
	Furthermore, if $(\pi,\pi_0)$ is a vertex of $(\epi(\subobj)-(x^*,\eta^*))^-$, then it satisfies the \MIS criterion.
\end{corollary}

Note that the reverse direction of the last sentence is generally not true, i.e. there might be minimal infeasible subsystems that do not correspond to vertices of the reverse polar set (see, \eg, \cref{ex:alt_poly_rev_polar}).

\subsection{Facet-defining Cuts}
\label{sec:gen_facet}

In cutting plane algorithms for polyhedra, facet-defining cuts are commonly considered to be
very useful since they form the smallest family of inequalities which completely describe the target polyhedron. A cutting-plane algorithm that can separate (distinct) facet inequalities in each iteration is not necessarily computationally efficient, but at least it is automatically guaranteed to terminate after a finite number of iterations. Also in practical applications, facet cuts have turned out to be extremely useful, e.g. in the context of branch-and-cut algorithms for integer programs such as the Traveling Salesman Problem. This is why the description of facet-defining inequalities has been a large and very active area of research for decades (see \cites{Balas:1975}{Nemhauser:1988}{Cook:1998}{Korte:2008} and, as mentioned before, \cite{Conforti:2018}).

Remember that a halfspace $\halfspace{\pi}{\alpha}$ is facet-defining for a set $C$ if $C \subseteq \halfspace{\pi}{\alpha}$ and $\hplane{\pi}{\alpha} \cap C$ contains $\dim(C)$
many affinely independent points.
Analogously to the \MIS criterion above, we define the \facet criterion for a normal vector in the context of Benders decomposition as follows:

\begin{definition}
	Let $\subobj$ be defined as in \cref{eq:gen_eta} and $(\pi,\pi_0) \in \reals^n \times \reals \setminus \Set{0}$. We say that $(\pi,\pi_0)$ \emph{satisfies the \facet criterion} if there exists $\alpha \in \reals$ such that $\halfspace{(\pi,\pi_0)}{\alpha}$ is either facet-defining for $\epi(\subobj)$ or the corresponding hyperplane $\hplane{(\pi,\pi_0)}{\alpha}$ contains $\epi(\subobj)$.
\end{definition}

Note that, in deviation from the common definition of a facet-defining cut, the above definition requires that the halfspace supports \emph{at least} $\dim(C)$ affinely independent points. In other words, in the case where $\epi(\subobj)$ is not full-dimensional, we also allow that $\epi(\subobj)$ is entirely contained in the hyperplane which represents the boundary of $\halfspace{(\pi,\pi_0)}{\alpha}$. In this situation, the comparison of different cut normals is inherently difficult: Since there is no clear way to tell if a cut supporting a facet of $\epi(\subobj)$ or one fully containing the set is the stronger cut, the \facet criterion captures arguably the strongest statement about a cut in relation to $\epi(\subobj)$ that we can make in general: In no case would we want to select a cut that supports \emph{neither} a facet \emph{nor} fully contains the set $\epi(\subobj)$.

The following result was originally obtained by \textcite[Theorem 4.5]{Balas:1998} in his analysis of disjunctive cuts. It reappears in \textcite[Theorem 6.2]{Cornuejols:2006}, using more familiar notation, but the latter contains a minor error in the case where the set $P$ is subdimensional. We therefore re-state a corrected version of the important parts of \cite[Theorem 6.2]{Cornuejols:2006} below, a corresponding proof can be found in \cite[Theorem 3.30]{Stursberg:2019}.

\begin{theorem}\label{thm:rev_polar_facets}
	Let $P \subseteq \reals^n$ be a polyhedron, $x^* \notin P$ and
	\begin{equation}
		r := \begin{cases}
				\dim(P)-1,	& x^* \in \aff(P)\\
				\dim(P),	& x^* \notin \aff(P).
			\end{cases}
	\end{equation}

Then, there exists an $x^*$-separating halfspace with normal vector $\pi \neq 0$ supporting an $r$-dimensional face of $P$ if and only if there exists a vertex $\pi^*$ of $\lin(P-x^*) \cap (P-x^*)^-$ and some $\lambda > 0$ such that $\lambda \pi \in \pi^* + \lin(P-x^*)^\bot$.
\end{theorem}

Most notably, for the case where $P$ is full-dimensional (\ie, $\dim(P)=n$) the above theorem implies that there exists an $x^*$-separating halfspace with normal vector $\pi$ supporting a facet of $P$ if and only if there exists a vertex $\pi^*$ of $(P-x^*)^-$ and $\lambda \geq 0$ such that $\lambda \pi = \pi^*$.

In this case, every cut generated from a vertex of the reverse polar set defines a facet of $\epi(\subobj)$. If an explicit $\CH$-representation of the reverse polar set is available, we can thus easily obtain a facet-defining cut, \eg, by linear programming.

Note that since $P^\leq(x^*,\eta^*)$ is line-free, \cref{lem:P_alt_reverse_polar} implies that for every vertex of the reverse polar set there exists a vertex of the relaxed alternative polyhedron (and hence of the original alternative polyhedron) that leads to the same cut normal. In other words, if the normal of an $x^*$-separating halfspace satisfies the \facet criterion, then it also satisfies the \MIS criterion.

On the other hand, \cref{lem:P_alt_reverse_polar} is not sufficient to guarantee that selecting a vertex of the alternative polyhedron yields a facet-defining cut: As \cref{ex:alt_poly_rev_polar} shows, a vertex of $P^\leq(x^*,\eta^*)$, is not necessarily mapped to a vertex of the reverse polar set under the transformation from \cref{lem:P_alt_reverse_polar}.
This exposes a useful hierarchy of subsets of the alternative polyhedron according to the properties of the cut normals which they yield: It is easy to select a vertex of the alternative polyhedron, which guarantees that the resulting cut normal satisfies the \MIS criterion, while
the points that lead to cut normals satisfying the \facet criterion constitute a subset of these vertices.
The approach of selecting \MIS-cuts may thus be viewed as a heuristic method to find \facet-cuts.

Although cuts satisfying the \MIS criterion in general do not satisfy the \facet criterion, we can obtain some information on when this is the case in the situation of \cref{thm:dual_problem_equiv}, \ie\ if the objective function $(\tilde\omega,\tilde\omega_0)$ used to select the cut via problem \cref{eq:dual_problem_equiv_alt} satisfies $(\tilde\omega,\tilde\omega_0) = (H\omega,-\omega_0)$ for some valid objective $(\omega,\omega_0)$ for problem \cref{eq:dual_problem_equiv_rev}.

In this case it turns out that we actually obtain a \facet-cut for all objectives $(\omega,\omega_0)$ except those from a lower-dimensional subspace. More precisely, we can prove the following characterization of the relationship between extremal points of the alternative polyhedron and cut normals satisfying the \facet criterion. This characterization is similar to
\cite[Proposition 6]{Conforti:2018}:

\begin{theorem}\label{thm:gen_vertex_equivalence}
	Let $\subobj$ be defined as in \cref{eq:gen_eta}, $(x^*, \eta^*) \in \reals^n \times \reals \setminus \epi(\subobj)$,
	and $(\omega,\omega_0) \in \cl(\pos(\epi(\subobj) - (x^*,\eta^*)))$.
  Then, there exists an optimal vertex $(\gamma^*,\gamma_0^*) \in P^\leq(x^*,\eta^*)$ with respect to the objective function $(H\omega,-\omega_0)$
  such that the resulting cut normal $(H^\top\gamma^*,-\gamma_0^*)$ is $(x^*,\eta^*)$-separating and satisfies the \facet criterion.
\end{theorem}

\begin{proof}
	Let $L := \lin(\epi(\subobj)-(x^*,\eta^*))$ and observe that $L$ is orthogonal to the lineality space of $(\epi(\subobj) -(x^*,\eta^*))^-$.
	From \cref{thm:cornuejols_reverse_gauge}, the reverse polar set $(\epi(\subobj) -(x^*,\eta^*))^-$ is bounded in the direction of $(\omega,\omega_0)$. We may therefore choose an optimal solution $(\pi,\pi_0)$ from the intersection $(\epi(\subobj) -(x^*,\eta^*))^- \cap L$. While the reverse polar need not be line-free, note that $(\epi(\subobj) -(x^*,\eta^*))^- \cap L$ is indeed line-free and we can therefore choose $(\pi,\pi_0)$ to be extremal in $(\epi(\subobj) -(x^*,\eta^*))^- \cap L$. By \cref{thm:dual_problem_equiv}, there exists $\gamma'$ with $H^\top \gamma' = \pi$ such that $(\gamma',-\pi_0)$ is an optimal solution to the problem
	\begin{equation}
		\max \Set*{(H \omega)^\top \gamma - \omega_0 \gamma_0 \given (\gamma, \gamma_0) \in P^\leq(x^*,\eta^*)}.\label{eq:gen_vertex_equivalence}
	\end{equation}
	Denote by $P^*$ the face of optimal solutions of \cref{eq:gen_vertex_equivalence} and observe that
	\begin{align*}
		(\gamma',-\pi_0) &\in P^* \cap \Set{(\gamma,\gamma_0) \given (H^\top \gamma, -\gamma_0) - (\pi,\pi_0) = 0}\\ &\subseteq P^* \cap \Set{(\gamma,\gamma_0) \given (H^\top \gamma, -\gamma_0) - (\pi,\pi_0) \in L^\bot}.
	\end{align*}
	Let $(\gamma^*,\gamma^*_0)$ be an extremal point of $P^* \cap \Set{(\gamma,\gamma_0) \given (H^\top \gamma, -\gamma_0) - (\pi,\pi_0) \in L^\bot}$ (which exists, since $P^\leq(x^*,\eta^*)$ is line-free). Then $(\gamma^*,\gamma^*_0)$ is obviously optimal for \cref{eq:gen_vertex_equivalence} and furthermore $(H^\top\gamma^*,-\gamma_0^*) = (\pi,\pi_0) + v$ with $v \in L^\bot$, which means by \cref{thm:rev_polar_facets} that it satisfies the \facet criterion.

  It remains to show that $(\gamma^*,\gamma^*_0)$ is a vertex of $P^*$, thus showing that it is also a vertex of $P^\leq(x^*,\eta^*)$. To see this,
  let $(\gamma^1,\gamma^1_0), (\gamma^2,\gamma^2_0) \in P^*$ such that $(\gamma^*,\gamma^*_0) \in \relint([(\gamma^1,\gamma^1_0),(\gamma^2,\gamma^2_0)])$. However, it follows that
  \[(\pi,\pi_0) + v = (H^\top\gamma^*,-\gamma^*_0) \in \relint([(H^\top\gamma^1,-\gamma^1_0),(H^\top\gamma^2,-\gamma^2_0)])\] and by \cref{lem:P_alt_reverse_polar}, $[(H^\top\gamma^1,-\gamma^1_0),(H^\top\gamma^2,-\gamma^2_0)] \subseteq (\epi(\subobj) -(x^*,\eta^*))^-$. As $(\pi,\pi_0)$ is extremal in the set $(\epi(\subobj) -(x^*,\eta^*))^- \cap L$, this implies that both $(H^\top\gamma^1,-\gamma^1_0),(H^\top\gamma^2,-\gamma^2_0) \in (\pi,\pi_0) + L^\bot$ which in turn implies that $(\gamma^1,\gamma^1_0), (\gamma^2,\gamma^2_0) \in P^* \cap \Set{(\gamma,\gamma_0) \given (H^\top \gamma, -\gamma_0) - (\pi,\pi_0) \in L^\bot}$. As $(\gamma^*,\gamma^*_0)$ is extremal in $P^* \cap \Set{(\gamma,\gamma_0) \given (H^\top \gamma, -\gamma_0) - (\pi,\pi_0) \in L^{\bot}}$, we obtain that $(\gamma^1,\gamma^1_0)=(\gamma^2,\gamma^2_0)=(\gamma^*,\gamma^*_0)$, which proves extremality of $(\gamma^*,\gamma^*_0)$ in $P^*$.
\end{proof}

In particular, the above theorem implies the following: If $(\gamma',\gamma_0') \in P^\leq(x^*,\eta^*)$ is an optimal extremal point with respect to the objective function $(H\omega,-\omega_0)$ such that the resulting cut normal $(H^\top\gamma',-\gamma_0')$ does not satisfy the \facet criterion, then the optimal solution for maximizing $(H\omega,-\omega_0)$ over $P^\leq(x^*,\eta^*)$ is not unique. Furthermore, by \cref{thm:dual_problem_equiv}, this implies that the same is true for maximizing the objective $(\omega, \omega_0)$ over $(\epi(\subobj) -(x^*,\eta^*))^-$.

We can summarize our results as follows:
While any \facet-cut is also an \MIS-cut, the reverse is not always true. However, if we
optimize the objective $(H\omega,-\omega_0)$ over the alternative polyhedron, then there exists only a subdimensional set of choices for the vector $(\omega,\omega_0)$ for which the resulting cut might not satisfy the \facet criterion (those, for which the optimum over the reverse polar set is non-unique).

This suggests that these cases should be \enquote{rare} in practice, especially if we choose (or perturb) $(\omega,\omega_0)$ randomly from some full-dimensional set. This argument, why a cut obtained for a generic vector $(\omega,\omega_0)$ can be expected to be facet-defining, is identical to the concept of \enquote{almost surely} finding facet-defining cuts proposed by \textcite{Conforti:2018}.

Looking back at \cref{rem:gen_alt_rep_rev_polar}, this similarity should not come as a surprise: With $(\omega, \omega_0) = (\bar x - x^*, \bar \eta - \eta^*)$ for a point $(\bar x, \bar \eta) \in \relint(\epi(\subobj))$, the resulting cut-generating LP is almost identical. In fact, the point $(\bar x, \bar \eta)$ in this case takes the role of the point that the origin is relocated into in the approach from \cite{Conforti:2018}.
Observe, however, that while \textcite{Conforti:2018} require that point to lie in the relative interior of $\epi(\subobj)$, we can actually expect a cut satisfying the \facet criterion from any $(\omega, \omega_0)$ for which the optimal objective over the reverse polar is strictly negative. By \cref{thm:cornuejols_reverse_gauge}, one sufficient (but not necessary) criterion for this is to choose $(\omega, \omega_0) = (\bar x - x^*, \bar \eta - \eta^*)$ for an arbitrary point $(\bar x, \bar \eta) \in \relint(\epi(\subobj))$.
%

\subsection{Pareto-Optimality}
\label{sec:gen_pareto}
The first systematic work on the general selection of Benders cuts to our knowledge was undertaken by \textcite{Magnanti:1981}. The paper, which has proven very influential and still being referred to regularly, focuses on the property of Pareto-optimality. It can intuitively be described as follows: A cut is Pareto-optimal if there is no other cut valid for $\epi(\subobj)$ which is \emph{clearly superior}, which \emph{dominates} the first cut.

In this setting, any cut that does not support $\epi(\subobj)$ is obviously dominated. Between supporting cuts, there is no general criterion for domination. We can, however, compare cuts where the cut normal $(\pi,\pi_0)$ satisfies $\pi_0 \neq 0$ (this is also the case covered by \cite{Magnanti:1981}):
\begin{definition}\label{def:pareto}
	For a problem of the form \cref{eq:gen_problem}, we say that an inequality $\pi^\top x +\pi_0 \eta \leq \alpha$ with $\pi_0 < 0$ is \emph{dominated} by another inequality $\pi'^\top x + \pi'_0 \eta \leq \alpha'$ if $\pi'_0 < 0$ and
	\begin{equation}
		\frac{\pi'^\top x-\alpha'}{-\pi'_0} \geq \frac{\pi^\top x-\alpha}{-\pi_0} \quad \text{for all $x \in S$,}
	\end{equation}
	with strict inequality for at least one $x \in S$.

	If $\pi_0 < 0$ and $\pi^\top x +\pi_0 \eta \leq \alpha$ is not dominated by any valid inequality for $\epi(\subobj)$, then we call it \emph{Pareto-optimal}.
\end{definition}

Remember that the set $S$ contains all points $x \in \reals^n$ that are feasible for an optimization problem of the form \cref{eq:gen_problem} if we ignore the linear constraints $Hx + Ay \leq b$. By the above definition, a cut dominates another cut if the minimum value of $\eta$ that it enforces is at least as good for all $x \in S$ and strictly better for at least one $x \in S$ (see \cref{fig:gen_cut_pareto}).

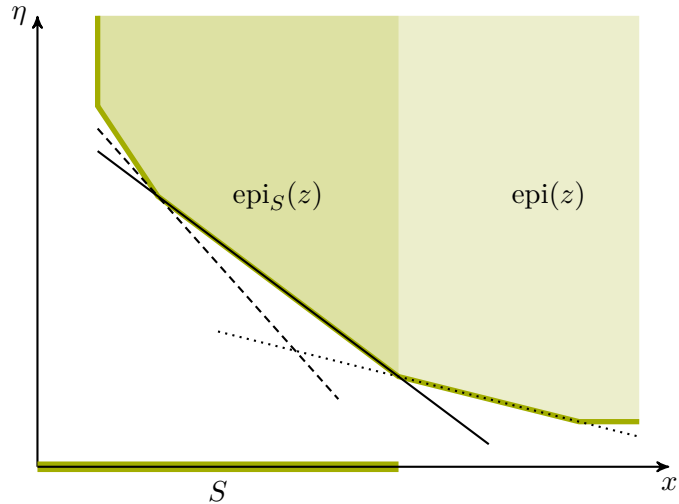
\begin{figure}
	\centering
	\begin{tikzpicture}[yscale=1.5, scale=0.8]
		\draw[thick, ->, >=stealth'] (0,0) -- (0,5) node[left] {$\eta$};
		\draw[line width=4pt, draw=accentuating green] (0,0) -- (6,0) node[below, midway] {$S$};
		\draw[thick, ->, >=stealth'] (0,0) -- (10.5,0) node[below] {$x$};

		\fill[highlight] (1,5) -- (1,4) -- (2,3) -- (4,2) -- (6,1) -- (9,0.5) -- (10,0.5) -- (10,5) -- cycle;
		\fill[highlight] (1,5) -- (1,4) -- (2,3) -- (4,2) -- (6,1) -- (6,5) -- cycle;
		\draw[line width=2pt, accentuating green] (1,5) -- (1,4) -- (2,3) -- (4,2) -- (6,1) -- (9,0.5) -- (10,0.5);

		\node at (8.5,3) {$\epi(\subobj)$};
		\node at (4,3) {$\epi_S(\subobj)$};
		\draw[thick, dotted] (3,1.5) -- (10,{0.5-1/3*0.5});
		\draw[thick] (1,3.5) -- (7.5,0.25);
		\draw[thick, densely dashed] (1,3.75) -- (5,{3.75-4*0.75});
	\end{tikzpicture}
	\caption{The dotted cut supports a facet of $\epi(\subobj)$ and it supports $\epi_S(\subobj)$, but it is still not Pareto-optimal. The solid cut supports a facet of $\epi_S(\subobj)$ and is hence Pareto-optimal. The dashed cut is Pareto-optimal even though it does not support a facet of $\epi(\subobj)$ (or $\epi_S(\subobj)$).}
	\label{fig:gen_cut_pareto}
\end{figure}

Analogously to the previous criteria, we define the \pareto criterion for a cut normal:

\begin{definition}
	For a problem of the form \cref{eq:gen_problem} with $\subobj$ as defined as in \cref{eq:gen_eta}, let $(\pi,\pi_0) \in \reals^n \times \reals$. We say that $(\pi,\pi_0)$ \emph{satisfies the \pareto criterion} if there exists a scalar $\alpha \in \reals$ such that the inequality $\pi^\top x +\pi_0 \eta \leq \alpha$ is Pareto-optimal.
\end{definition}

This criterion is very reasonable: If a cut is not Pareto-optimal, then it can be replaced by a different cut, which is also valid for $\epi(\subobj)$, but leads to a strictly tighter approximation. We would hence prefer to generate a stronger, Pareto-optimal cut right away.

The following theorem provides us with a characterization of Pareto-optimal cuts. It is based on the idea of \cite[Theorem~1]{Magnanti:1981}, which is formulated under the assumption that the subproblem is always feasible (which implies that $\pi_0 < 0$ for any cut normal $(\pi,\pi_0)$). While the original theorem is only concerned with sufficiency, we extend the result in a natural way to obtain a criterion that gives
a complete characterization of Pareto-optimal cuts. We use the following separation lemma:

\begin{lemma}[\cite{Rockafellar:1970}]\label{thm:sep_not_contained}
	Let $C \subseteq \reals^n$ be a non-empty convex set and $K \subseteq \reals^n$ a non-empty polyhedron such that $\relint(C) \cap K = \emptyset$. Then, there exists a hyperplane separating $C$ and $K$ which does not contain $C$.
\end{lemma}

Using this lemma, we obtain the following theorem (for the proof, we refer to \cite[Theorem 3.40]{Stursberg:2019}):

\begin{theorem}\label{thm:gen_pareto_if_facet}
	For a problem of the form \cref{eq:gen_problem}, let $(\pi, \pi_0) \in \reals^n \times \reals$ with $\pi_0 < 0$. The inequality $\pi^\top x +\pi_0 \eta \leq \alpha$ is Pareto-optimal if and only if $\halfspace{(\pi,\pi_0)}{\alpha}$ is a halfspace supporting $\epi(\subobj)$ in a point $(x^*,\eta^*) \in \epi(\subobj) \cap \relint(\conv(S)) \times \reals$.
\end{theorem}

For the case where $S$ is convex, the previous theorem immediately implies the following statement:
\begin{corollary}
  Let $S$ be convex. Then, $\pi^\top x +\pi_0 \eta \leq \alpha$ is Pareto-optimal if and only if $\halfspace{(\pi,\pi_0)}{\alpha}$ supports a face $F$ of $\epi_S(\subobj)$ such that $F \not\subset \relbd(S) \times \reals$.
\end{corollary}

\Textcite{Magnanti:1981} also provide an algorithm that computes a Pareto-optimal cut by solving the cut-generating problem twice. While their algorithm is defined for the original Benders optimality cuts, it can be adapted to work with other cut selection criteria, as well. \Textcite{Sherali:2013} present a method based on multiobjective optimization to obtain a cut that satisfies a weaker version of Pareto-optimality by solving only a single instance of the cut-generating LP.
\Textcite{Papadakos:2008} notes that, given a point in the relative interior of $\conv(S)$, a Pareto-optimal cut can be generated using a single run of the cut-generating problem. Also, under certain conditions on the problem, other points not in the relative interior allow this, as well. However, the approach suggested by the authors adds Pareto-optimal cuts independently from master- or subproblem solutions, together with subproblem-generated cuts, which are generally not Pareto-optimal. This means that the Pareto-optimal cuts which are added may not even cut off the current tentative solution. The upcoming \cref{thm:gen_pareto_obj} will lead to an approach that reconciles both objectives, generating cuts that are both Pareto-optimal and cut off the current tentative solution.

We use a result by \textcite{Cornuejols:2006} on the set of points exposed by a cut normal $(\pi,\pi_0)$ to derive a method that always obtains a Pareto-optimal cut. The following lemma has been slightly generalized and rewritten to match our setting and notation, but it follows the general idea of \Textcite[Theorem 3.4]{Cornuejols:2006}.

\begin{lemma}\label{lem:gen_exposed}
	Let $(x^*,\eta^*) \in \reals^n \times \reals \setminus \epi(\subobj)$ and $(\omega, \omega_0) \in \pos(\epi(\subobj)-(x^*,\eta^*))$ and let $(\pi,\pi_0)$ be optimal in $Q := (\epi(\subobj)-(x^*,\eta^*))^-$ with respect to the objective $(\omega, \omega_0)$. Then there exists $\alpha \in \reals$ such that $\halfspace{(\pi,\pi_0)}{\alpha}$ supports $\epi(\subobj)$ in
	\begin{equation}
		(\bar x, \bar\eta) := \frac{(\omega, \omega_0)}{-h_Q(\omega,\omega_0)} + (x^*,\eta^*).
	\end{equation}
\end{lemma}

\begin{proof}
  The case of $(\omega, \omega_0) \in (\epi(\subobj)-(x^*,\eta^*))$ was proven by \Textcite[Theorem 3.4]{Cornuejols:2006}. If $(\omega, \omega_0) \in \pos(\epi(\subobj)-(x^*,\eta^*))$, then there is $\mu > 0$ such that $\mu \cdot (\omega,\omega_0) \in (\epi(\subobj)-(x^*,\eta^*))$. Note that if $(\pi,\pi_0)$ is optimal with respect to $(\omega, \omega_0)$, then also with respect to $\mu \cdot (\omega,\omega_0)$. Thus it follows from \Textcite[Theorem 3.4]{Cornuejols:2006} that there exists $\alpha \in \reals$ such that $H^\leq_{(\pi, \pi_0),\alpha}$ supports $\epi_S(\subobj)$ in
	\begin{equation}
		(\bar x, \bar\eta) := \frac{\mu \cdot (\omega, \omega_0)}{-h_Q(\mu\omega,\mu\omega_0)} + (x^*,\eta^*)
		= \frac{(\omega, \omega_0)}{-h_Q(\omega,\omega_0)} + (x^*,\eta^*).
	\end{equation}
\end{proof}

We can now prove the theorem already mentioned above.

\begin{theorem}\label{thm:gen_pareto_obj}
	Let $(x^*,\eta^*) \in S \times \reals$, $(\omega, \omega_0) \in \relint(\conv(\epi_S(\subobj)-(x^*,\eta^*)))$, $(\pi,\pi_0)$ be optimal in $(\epi(\subobj)-(x^*,\eta^*))^-$ with respect to the objective $(\omega, \omega_0)$, and $\pi_0 < 0$. Then $(\pi,\pi_0)$ satisfies the \pareto criterion.
\end{theorem}

\begin{proof}
	Let $Q := (\epi(\subobj)-(x^*,\eta^*))^-$ again and $\lambda := -(h_Q(\omega,\omega_0))^{-1}$. Since, in particular, $(\omega, \omega_0) \in \epi_S(\subobj)-(x^*,\eta^*)$ it follows from the definition of the reverse polar set that $h_Q(\omega,\omega_0) \leq -1$ and thus $\lambda \in [0,1]$.

  For $(\bar x, \bar\eta)$ from \cref{lem:gen_exposed}, we thus obtain that $(\bar x, \bar\eta) = \lambda \left((\omega,\omega_0) + (x^*,\eta^*)\right) + (1-\lambda) (x^*,\eta^*)$ is a convex combination of a point $(\omega, \omega_0) + (x^*,\eta^*) \in \relint(\conv(\epi_S(\subobj))) \subseteq \relint(\conv(S)) \times \reals$ and $(x^*,\eta^*) \in S \times \reals$. Therefore, $\bar x \in \relint(\conv(S))$ and thus by \cref{thm:gen_pareto_if_facet} the cut defined by $(\pi,\pi_0)$ is Pareto-optimal.
\end{proof}

\begin{table*}[t]

\centering
\begin{tabular}{l|ccc}
 $(\tilde\omega,\tilde\omega_0)$	&\MIS & \facet& \pareto\\ \hline
$\in \reals^m\times\reals$ & \cmark	& \xmark	& \xmark\\
$\in (H,-1)\cdot\reals^n\times\reals$ & \cmark	& (\cmark)	& \xmark\\
$\in (H,-1)\cdot\relint(\conv(\epi_S(\subobj)-(x^*,\eta^*)))$ & \cmark	& (\cmark)	& \cmark
\end{tabular}
\caption{Properties of a cut resulting from an extremal point in the alternative polyhedron which maximizes $(\tilde\omega,\tilde\omega_0)$ (under the assumption that a finite optimum exists). The checkmark in parentheses (\cmark) indicates that the property is satisfied for all $(\tilde\omega,\tilde\omega_0)$ in the specified set except those from a specific sub-dimensional subset.}\label{tab:cut_selection_properties}
\end{table*}

The results from this section are summarized in \cref{tab:cut_selection_properties}.

\section{Outlook}

We conclude this paper by an outlook on interesting research questions in the context of cut selection for Benders decomposition.

In a generic implementation of Benders decomposition, feasible solutions are used primarily to decide when the algorithm has converged. By \cref{thm:cornuejols_reverse_gauge}, any such solution can be used to derive a subproblem objective which satisfies the prerequisites of both \cref{thm:gen_opt_gamma} and \cref{thm:gen_vertex_equivalence}. They thus result in the generation of cuts which are always supporting and often even support a facet. In our preliminary experiments, these cuts proved to be very useful in improving the performance of a Benders decomposition algorithm. Since information from a feasible solution can thus be used within the cut generation, it makes sense to investigate more closely the possibilities how such a solution can be obtained during the algorithm. This is likely to be very problem-specific, but some general ideas could be:
	\begin{itemize}
		\item How is the information from feasible solutions computed in different iterations best aggregated? Does it make sense to use \eg~a stabilization approach or a convex combination with some other choices for $(\omega,\omega_0)$, \eg~from previous iterations? This corresponds to the method used by \textcite{Papadakos:2008} in their empirical study.
		\item More broadly, what different methods can be used to generate feasible solutions and what effect do different feasible solutions have on cut generation and the computational performance of the algorithm?
	\end{itemize}

Furthermore, if a feasible solution is not available as the basis for a subproblem objective, the cut-generating problem might be unbounded/infeasible. On the other hand, the approach from \cite{Fischetti:2010} with $\tilde \omega = \mathbbm{1}$ yields a cut-generating LP that is always feasible, but the resulting cut might be weaker. How can both approaches be combined in a best-possible way? For instance, is choosing $\tilde\omega = H\omega + \epsilon \cdot \mathbbm{1}$ as the relaxation term and letting $\epsilon$ go to zero a good choice?

Finally, our approach provides a clear geometric interpretation of the interaction between parametrization of the cut-generating LP and the resulting cut normals. How can this be used to leverage a-priori knowledge about the problem (or information obtained through a fast preprocessing algorithm) to improve the selection of a subproblem objective $(\omega,\omega_0)$ from a set of cuts satisfying the same quality criteria (e.g. that are all facet-defining)?

These questions would be worth addressing in the context of further theoretical research, but also more extensive empirical studies of different parametrization strategies.

\bigskip

\textit{Acknowledgements:} The authors thank their colleagues Magdalena Stüber and Michael Ritter for their invaluable support.

\printbibliography[heading=bibintoc]

\end{document}